\documentclass[a4,12pt,epsfig]{article}
\usepackage{graphicx}
\usepackage{subfig}
\usepackage{epsfig}
\usepackage[hyperfootnotes=false, colorlinks, linkcolor={blue}, citecolor={magenta}, filecolor={blue}, urlcolor={blue}]{hyperref}
\usepackage{amsmath}
\usepackage{amssymb}
\usepackage{amsfonts}
\usepackage{euscript}
\begin{document}
       
\bibliographystyle{plain}
\newcommand{\bea}{\begin{eqnarray}}
\newcommand{\eea}{\end{eqnarray}}
\newcommand{\bfmN}{{\mbox{\boldmath{$N$}}}}
\newcommand{\bfmx}{{\mbox{\boldmath{$x$}}}}
\newcommand{\bfmv}{{\mbox{\boldmath{$v$}}}}
\newcommand{\se}{\setcounter{equation}{0}}
\newtheorem{corollary}{Corollary}[section]
\newtheorem{example}{Example}[section]
\newtheorem{definition}{Definition}[section]
\newtheorem{theorem}{Theorem}[section]
\newtheorem{proposition}{Proposition}[section]
\newtheorem{lemma}{Lemma}[section]
\newtheorem{remark}{Remark}[section]
\newtheorem{result}{Result}[section]
\newcommand{\vtwo}{\vskip 4ex}
\newcommand{\vthree}{\vskip 6ex}
\newcommand{\vfour}{\vspace*{8ex}}
\newcommand{\hone}{\mbox{\hspace{1em}}}
\newcommand{\hon}{\mbox{\hspace{1em}}}
\newcommand{\htwo}{\mbox{\hspace{2em}}}
\newcommand{\hthree}{\mbox{\hspace{3em}}}
\newcommand{\hfour}{\mbox{\hspace{4em}}}
\newcommand{\von}{\vskip 1ex}
\newcommand{\vone}{\vskip 2ex}
\newcommand{\n}{\mathfrak{n} }
\newcommand{\m}{\mathfrak{m} }
\newcommand{\q}{\mathfrak{q} }
\newcommand{\aF}{\mathfrak{a} }

\newcommand{\kl}{\mathcal{K}}
\newcommand{\p}{\mathcal{P}}
\newcommand{\Lt}{\mathcal{L}}
\newcommand{\bv}{{\mbox{\boldmath{$v$}}}}
\newcommand{\bc}{{\mbox{\boldmath{$c$}}}}
\newcommand{\bx}{{\mbox{\boldmath{$x$}}}}
\newcommand{\br}{{\mbox{\boldmath{$r$}}}}
\newcommand{\bs}{{\mbox{\boldmath{$s$}}}}
\newcommand{\bb}{{\mbox{\boldmath{$b$}}}}
\newcommand{\ba}{{\mbox{\boldmath{$a$}}}}
\newcommand{\bn}{{\mbox{\boldmath{$n$}}}}
\newcommand{\bp}{{\mbox{\boldmath{$p$}}}}
\newcommand{\by}{{\mbox{\boldmath{$y$}}}}
\newcommand{\bz}{{\mbox{\boldmath{$z$}}}}
\newcommand{\be}{{\mbox{\boldmath{$e$}}}}
\newcommand{\proof}{\noindent {\sc Proof :} \par }
\newcommand{\bP}{{\mbox{\boldmath{$P$}}}}

\newcommand{\M}{\mathcal{M}}
\newcommand{\R}{\mathbb{R}}
\newcommand{\Q}{\mathbb{Q}}
\newcommand{\Z}{\mathbb{Z}}
\newcommand{\N}{\mathbb{N}}
\newcommand{\C}{\mathbb{C}}
\newcommand{\xar}{\longrightarrow}
\newcommand{\ov}{\overline}
 \newcommand{\rt}{\rightarrow}
 \newcommand{\om}{\omega}
 \newcommand{\wh}{\widehat }
 \newcommand{\wt}{\widetilde }
 \newcommand{\g}{\Gamma}
 \newcommand{\lm}{\lambda}

\newcommand{\eN}{\EuScript{N}}
\newcommand{\ncom}{\newcommand}
\newcommand{\norm}[1]{\left| #1 \right|}
\newcommand{\inp}[2]{\langle{#1},\,{#2} \rangle}
\newcommand{\nrms}[1]{\left\lVert #1 \right\rVert^2}
\newcommand{\nrm}[1]{\left\lVert #1 \right\rVert}

\title{An existence and uniqueness result using bounded variation estimates in Galerkin approximations
}
\author{  Ramesh Mondal\footnote{rmondal86@gmail.com} and Aditi Sengupta\footnote{aditisengupta54@gmail.com}	
\\
Mathematics Department, University of Kalyani,\\ Kalyani, Nadia-741235, West Bengal, India.}

\maketitle{}

\begin{abstract}
Bounded variation estimates of Galerkin approximations are established in order
to extract an almost everywhere convergent subsequence of Galerkin approximations. As a result we prove existence of weak solutions of initial boundary value problems for quasilinear parabolic equations.  Uniqueness of weak solutions is derieved applying a standard argument.	
\end{abstract}

\section{Introduction}
Let $\Omega$ be a bounded domain in $\R^d$ with 
smooth boundary $\partial \Omega$. For $T>0$, 
denote $\Omega_T := \Omega \times (0,T)$. 
Consider the initial boundary value problem
\begin{eqnarray}
u_t + \nabla . f(u)  
\!\!\!& = \!\!\!& 
\sum^d_{i,j=1} ( B_{ij}(u) u_{x_j})_{x_i} \phantom{m}\text{ in }\Omega_T
\label{eq1.1}
\\
u(x,t)
\!\!\!&=\!\!\!&  
0 \phantom{m} \text{ on } \partial \Omega \times (0,T)\label{eq1.2}
\\
u(x,0)
\!\!\!&=\!\!\!&  
u_0(x) \phantom{m} x \in \Omega,\label{eq1.3}
\end{eqnarray}
where $f:=\left(f_{1},f_{2},\cdots,f_{d}\right)$, $f_i : \R \to \R $ and  $B_{ij} : \R \to \R$ are given 
 functions and $u_0 : \Omega \to \R$ is initial data.

The aim of this article is to prove the existence of a weak solution of 
\eqref{eq1.1} applying bounded variation (BV) compactness of Galerkin approximations in the standard method of Galerkin approximations. The main difficulty in showing the existence of a weak solution is to show the existence of an almost everywhere convergent subsequence of Galerkin approximations. This extraction of almost everywhere convergent subsequence is the consequence of BV compactness. As a result, we give an alternative proof of the standard argument of extraction of almost everywhere convergent subsequence of Galerkin approximations described below. We point out that the application of BV comactness in the context of Galerkin approximations is new to the best of our knowledge.  \\
\vspace{0.1cm}\\
The equation of the form \eqref{eq1.1} appears in many physical applications. For instances, \eqref{eq1.1} appears in electromagnetic plane field problems (see applications in \cite{Axelsson}). The equation \eqref{eq1.1} with $f=0$ also appears in water flooding problem in petroleum engineering and thermal conduction problem in models for chemical reactions \cite{Garcia} and two phase flows in a porous medium such as the air-water flow of hydrological aquifers \cite{Eymard}. \\
\vspace{0.1cm}
The equation of the form \eqref{eq1.1} is also important in the study of numerical analysis. Axelsson \cite{Axelsson} establishes  asymtotic error estimates of Galerkin approximations of \eqref{eq1.1} which are valied for all $t>0$ under additional assumptions in the domain $\Omega\times (0,\infty)$. These estimates are also valid for the corresponding elliptic stationary problems in divergence form. The method which is used in \cite{Axelsson} is based on the classical technique using elliptic projections and the negative norm.\\
In \cite{Douglas}, Douglas {\it et al.} consider a numerically efficient modification of an extrapolated Crank-Nicolson-Galerkin method to approximate a solution of quasilinear parabolic equation with Neumann boundary condition. Their modification is based on a preconditioned conjugate gradient iteration on a single time step. As a consequence, they derive a global error estimate using the local error estimates already derived for the conjugate gradient procedure. In this case, their modified method preserves the accuracy inherent in the underlying Crank-Nicolson-Galerkin method. \\ 

We now want to impose hypothesis on the flux $f$, the diffusion coefficient $B_{ij}$ 
and the initial data $u_0$. To do this, let us fix a few notations. 
Let  $k \in \N$ and the function space $(C^k(\R))^d$ denotes the space of $d$-tuples of 
$k$-times continuously differentiable functions 
and the standard function space $(L^\infty(\R))^d$ denotes the $d$-tuples of 
essentially bounded functions equipped with the norm
$$
\nrm f_{\left(L^\infty (\R)\right)^{d}}
=
\displaystyle\max_{1 \leq i \leq d}\,\left(\displaystyle\sup_{x \in \Omega} |f_i(x)|\right).
$$ 
We also denote $\nrm B_{L^{\infty}(\R)} = \displaystyle\max_{1 \leq i,j \leq d} \nrm {B_{ij}}_{L^{\infty}(\R)}$ 
and we refer the reader to \cite{Kesavan} for the spaces $H^{2}(\Omega)$ and $H^{1}_{0}(\Omega)$. The conditions on the flux function $f$, diffusion coefficients $B_{ij}$ and initial data $u_0$ are as follows.

\noindent
\underline{Hypothesis-H1}
\begin{enumerate}
\item[(HF)]
$f \in (C^2(\R))^d~,~ f^{\prime} \in (L^\infty(\R))^d$.
\item[(HB)]
$B_{ij} \in C^2(\R)\cap L^\infty(\R)$ and there exists a real number $\theta>0$ such that the matrix B = $(B_{ij})$ satisfies 
$$
\sum_{i,j=1}^d B_{ij} \xi_i \xi_j \ge \theta\,\left|\xi\right|^{2}
\phantom{m} \text{ for all }~ 
\xi \in \R^d.
$$
\item[(HI)]
$u_0 \in H^{2}(\Omega)\cap H^{1}_{0}(\Omega)$.
\end{enumerate}
\noindent
\underline{Hypothesis-H2}
\begin{enumerate}
	\item[(HF)]
	$f \in (C^2(\R))^d~,~ f^{\prime} \in (L^\infty(\R))^d$.
	\item[(HB)]
	$B_{ij} \in C^2(\R)\cap L^\infty(\R)$ and there exists a real number $\theta>0$ such that the matrix B = $(B_{ij})$ satisfies 
	$$
	\sum_{i,j=1}^d B_{ij} \xi_i \xi_j \ge \theta\,\left|\xi\right|^{2}
	\phantom{m} \text{ for all }~ 
	\xi \in \R^d.
	$$
	\item[(HI)]
	$u_0 \in L^{2}\left(\Omega\right)$.
\end{enumerate}

In this article, we prove the following three main results. The differences of the first two results are in the assumption of initial data and the arguments to prove them. We apply the famous technnique of Bardos {\it et al.} \cite{Bardos} to prove the following result.
\begin{theorem}\label{Aditi_RM_Convergence_result_1}
Let $f$, $\left(B_{ij}\right)$ and $u_{0}$ satisfy \textbf{Hypothesis-H1}. Then there exists a unique solution $u$ in $W(0,T)$ of the initial boundary value problem \eqref{eq1.1}, \eqref{eq1.2}, \eqref{eq1.3} which satisfies \eqref{Aditi_RM_equation1_Defn} and  \eqref{Aditi_RM_equation2_Defn}. 
\end{theorem}
The motivation to prove Theorem \ref{Aditi_RM_Convergence_result_1} is to observe the best possible result which can be obtained by applying Bardos {\it et al.} \cite{Bardos} technique.
Then, we apply Hahn Banach extension theorem and Riesz representation theorem to conclude the following result.
\begin{theorem}\label{Aditi_RM_Convergence_result_2}
Let $f$, $\left(B_{ij}\right)$ and $u_{0}$ satisfy \textbf{Hypothesis-H2}. Then there exists a unique solution $u$ in $W(0,T)$ of the initial boundary value problem \eqref{eq1.1}, \eqref{eq1.2}, \eqref{eq1.3} which satisfies \eqref{Aditi_RM_equation1_Defn} and  \eqref{Aditi_RM_equation2_Defn}. 
\end{theorem}
In the next result, we show that the solution of \eqref{eq1.1}, \eqref{eq1.2}, \eqref{eq1.3} satisfies the following estimates.
\begin{theorem}\label{Aditi_RM_Convergence_result_3}
Let $u_{0}$ be in $ L^{\infty}\left(\Omega\right)$ and $u$ be the unique solution of \eqref{eq1.1}, \eqref{eq1.2}, \eqref{eq1.3}. Then $u$ satisfies $\left\|u\right\|_{L^{\infty}\left(\Omega\right)}\leq \|u_{0}\|_{L^{\infty}\left(\Omega\right)}$ for almost every $t$ in $\left(0,T\right)$.		
\end{theorem} 
The function space $W(0,T)$ is defined in Section \ref{Existence_of_Solution}.  We now mention the standard arguments to show the existence of an almost everywhere convergent subsequence of Galerkin approximations from \cite[p.469]{Ladyzenskaja}. Firstly, energy estimates, Arzela-Ascoli theorem and a diagonal argument are used to extract a convergent subsequence of coefficients of Galerkin approximations in $C([0,T])$. Then a standard inequality regarding $H^{1}\left(\Omega\right)$ functions from \cite[p.72]{Ladyzenskaja} is used to the Galerkin approximations in order to show that Galerkin approximations are convergent in $L^{2}\left(0,T; L^{2}\left(\Omega\right)\right)$. As a consequence of these arguments, we get almost everywhere convergent subsequence of Galerkin approximations. An existence result using the method of Galerkin approximations is derived in \cite{Ramesh} by following the above mentioned  method of standard argument regarding extraction of almost everywhere convergent subsequence of Galerkin approximations. In this article, we establish bounded variation estimates (BV-estimates) of Galerkin approximations and use the fact that $BV\left(\Omega_{T}\right)$ is compactly imbedded in $L^{1}\left(\Omega_{T}\right)$ in order to show the existence of almost everywhere convergent subsequence of Galerkin approximations. Therefore we give an alternative proof of the standard arguments as mentioned above of showing the existence of almost everywhere convergent subsequence of Galerkin approximations. The BV-estimates of Galerkin approximations involves the $L^{1}\left(\Omega_{T}\right)-$ estimates for the time derivative and space derivatives of the Galerkin approximations. The main difficulty in proving the BV-estimate is to get the $L^{1}\left(\Omega_{T}\right)-$ estimates for the time derivative of the Galerkin approximations. In order to get $L^{1}\left(\Omega_{T}\right)-$ estimates for the time derivative of the Galerkin approximations, we apply the idea used by {\it Bardos et al.} in \cite{Bardos} to prove Theorem \ref{Aditi_RM_Convergence_result_1} for $u_{0}$ in $H^{2}\left(\Omega\right)\cap H^{1}_{0}\left(\Omega\right)$. But the existence of a weak solution is proved for $u_{0}$ in $L^{2}\left(\Omega\right)$ using standard argument as mentioned in \cite[p.469]{Ladyzenskaja}. Then, we apply Hahn Banach extension theorem and Riesz represtation theorem to prove Theorem \ref{Aditi_RM_Convergence_result_2} for $u_{0}$ in $L^{2}\left(\Omega\right)$. The arguments used in the proof of Theorem \ref{Aditi_RM_Convergence_result_2} is new and it completes the alternative proof of the standard arguments used in \cite[p.469]{Ladyzenskaja} in the sense of initial data. We believe that our arguments used in showing the existence of an almost everywhere convergent subsequence in Theorem \ref{Aditi_RM_Convergence_result_2} is much easier to understand and smarter than the standard arguments used in \cite[p.469]{Ladyzenskaja}. \\
In the context of viscosity approximations to scalar conservation laws, the idea used by {\it Bardos et al.} in \cite{Bardos} is to differentiate the equation with respect to $t$, multiply the resultant by a sequence of multipliers which are approximations of signum function and  integrate the equation over $\Omega$. Then the $L^{1}\left(\Omega_{T}\right)$ estimates of time derivative is established by applying the properties of the sequence of multipliers and careful computation whenever $u_{0}$ lies in $H^{1}\left(\Omega\right)$. In \cite{Ramesh}, {\it R. Mondal et al.} uses a different approximations of signum function. \\
Let $\left(u_{m}(t)\right)$ be the sequence of Galerkin approximations for $m\in\mathbb{N}$ and $t\in[0,T]$ defined in Subsection \ref{Subsection_ Galerkin_approximation}. In order to establish the $L^{1}\left(\Omega_{T}\right)-$ estimates for time derivative of Galerkin approximations, we apply the idea used by {\it Bardos et al.} in \cite{Bardos} with a different sequence of approximations $\mbox{tan\,h}\left(n\,\frac{\partial u_{m}}{\partial t}\right)$ of signum function. The sequence of approximations $\mbox{tan\,h}\left(n\,\frac{\partial u_{m}}{\partial t}\right)$ of signum function is originally used in \cite{Ramesh} in order to establish BV-estimates of viscous approximations to scalar conservation laws. As a consequence of the application of the idea used by {\it Bardos et al.} in \cite{Bardos}, we are in need to get the $L^{2}\left(\Omega\right)$ estimates of first and second order derivatives of the Galerkin approximations $\left(u_{m}(0)\right)$. These estimates are obtained using the facts that $\|u_{m}(0)\|_{H^{1}_{0}\left(\Omega\right)}\leq \|u_{0}\|_{H^{1}_{0}\left(\Omega\right)}$  and there exists a constant $E>0$ such that $\|u_{m}(0)\|_{H^{2}\left(\Omega\right)}\leq E\,\|u_{0}\|_{H^{2}\left(\Omega\right)}$. We refer the reader to Theorem \ref{BV} for details. We point out that the application of the idea used in {\it Bardos et al.} to get $L^{1}\left(\Omega_{T}\right)$ estimates for time derivatives of Galerkin approximations is new in the context of Galerkin approximations to the best of our knowledge in Theorem \ref{Aditi_RM_Convergence_result_2}. The $L^{1}\left(\Omega_{T}\right)$  estimates of time derivative is a consequence of applications of Energy estimates, Hahn Banach extension theorem and Riesz represtation theorem in Theorem \ref{Aditi_RM_Convergence_result_2}. The $L^{1}\left(\Omega_{T}\right)$ estimates for the space derivatives is a consequence of Energy estimates given in Subsection \ref{Energy_Estimates}. Then almost everywhere convergent subsequence of the Galerkin approximations $\left(u_{m}\right)$ is obtained using the standard arguments that $\mbox{BV}\left(\Omega_{T}\right)$ is compactly imbedded in $L^{1}\left(\Omega_{T}\right)$ and application of this compact imbedding result in the context of Galerkin approximation is also new.\\
In \cite{Antontsev}, Antontsev and \"Ozt\"urk study the nonlinear boundary value problem for the p-Laplacian equations in bounded domain. They use Faedo-Galerkin approximation together with a combinations of compactness and monotonicity method to prove the existence of solution. We now compare the computation of our $L^{1}\left(\Omega_{T}\right)$ time derivative estimates for Galerkin approximations with \cite{Antontsev}. In \cite{Antontsev}, Antontsev and \"{O}zt\"{u}rk multiply the projection of the p-Laplace equation on $\mbox{span}\left\{w_{1},w_{2},\cdots,w_{k}\right\}$ by the multipliers $u_{m}^{\prime}$ and integrate over $\Omega$. Then $L^{2}\left(\Omega_{T}\right)$ estimates of time derivatives for Galerkin approximations are obtained by applying embedding result, several standard inequalities, the fact that $\|\nabla u_{m}(0)\|_{L^{p}\left(\Omega\right)}\leq \|\nabla u_{0}\|_{L^{p}\left(\Omega\right)}$ and clever computations.  In Theorem \ref{Aditi_RM_Convergence_result_2}, since we apply the idea used by {\it Bardos et al.} in \cite{Bardos}, we get $L^{1}\left(\Omega_{T}\right)$ estimates of the time derivative instead of $L^{2}\left(\Omega_{T}\right)$ estimates. In the computation,  we need to use $\|u_{m}(0)\|_{H^{1}_{0}\left(\Omega\right)}\leq \|u_{0}\|_{H^{1}_{0}\left(\Omega\right)}$ which is similar to the estimates $\|\nabla u_{m}(0)\|_{L^{p}\left(\Omega\right)}\leq \|\nabla u_{0}\|_{L^{p}\left(\Omega\right)}$ of \cite{Antontsev}. But we also need to use $\|u_{m}(0)\|_{H^{2}\left(\Omega\right)}\leq E\,\|u_{0}\|_{H^{2}\left(\Omega\right)}$ in our computation of $L^{1}\left(\Omega_{T}\right)$ estimates for the time derivative of Galerkin approximations because of the presence of $L^{2}\left(\Omega\right)$ estimates of the second derivatives of Galerkin approximations $\left(u_{m}(0)\right)$.
We refer the reader to \cite{Lefton} for another result on Faedo-Galerkin approximations and Compactness argument. We also refer the reader to \cite{Chen} and \cite{Kobayasi} for results on existence and uniqueness of solution for parabolic problems. \\
The plan for the paper is the following. In Subsection \ref{Subsection_ Galerkin_approximation}, we prove the existence of Galerkin approximations. In Subsection \ref{Energy_Estimates}, we derive the energy estimates of the Galerkin approximations. In Subsection \ref{BV_Estimates}, we establish BV-estimates of Galerkin approximations and show the existence of almost everywhere convergent subsequence of Galerkin approximations using Bardos {et al.}'s idea. In Subsection \ref{Convergence}, we prove Theorem \ref{Aditi_RM_Convergence_result_1}, Theorem \ref{Aditi_RM_Convergence_result_2} and also derive $L^{\infty}\left(\Omega\right)$ estimate for the weak solution.


\section{Existence of solution}\label{Existence_of_Solution}
We now prove the existence of solution of \eqref{eq1.1} using Galerkin approximation in the function space
$$
W(0,T)
:=
\left\{ u \in L^2 \big( 0,T; H^1_0(\Omega) \big) : 
u_t \in L^2 \big( 0,T; H^{-1}(\Omega) \big) \right\}.
$$
\begin{definition}$\left(\textbf{\mbox{Weak Solution}}\right)$ 
Let $u \in W(0,T)$. Then $u$ is called a weak solution of \eqref{eq1.1} 
if for almost every $0\le t \le T$ and for all $ v \in H^1_0(\Omega)$, $u$ satisfies 
\begin{eqnarray}
\inp {u_t}v
 + 
 \sum_{i=1}^d \sum_{j=1}^d \int_{\Omega}B_{ij}(u)
 \frac{\partial u}{\partial x_j} \frac{\partial v}{\partial x_i} dx 
 + 
 \sum_{i=1}^d \int_{\Omega} f^{\prime}_i(u) 
 \frac{\partial u}{\partial x_i}v dx
 \!\!\!& =\!\!\!&
 0\label{Aditi_RM_equation1_Defn}
\\
u(.,0) \!\!\!& =\!\!\!& u_0(.),\label{Aditi_RM_equation2_Defn}
\end{eqnarray}
where $\inp..$ denotes the duality pairing between 
$H^{-1}(\Omega)$ and $H^1_0(\Omega)$.
\end{definition}
\vspace{0.1cm}
In the method of Galerkin approximations, we prove the following four steps.
\\
Step 1 : Construction of sequence of Galerkin approximations.
\\
Step 2 : Energy estimates of Galerkin approximations .
\\
Step 3 : BV-estimates of Galerkin approximations.
\\
Step 4 : Finally, we show that the limit of the sequence of Galerkin approximations is a weak solution of problem \eqref{eq1.1}, \eqref{eq1.2}, \eqref{eq1.3} and uniqueness of weak solutions is concluded form \cite[p.150]{Ladyzenskaja}.\\
All the above four steps are proved in the following subsections.

\subsection{Existence of Galerkin approximations}\label{Subsection_ Galerkin_approximation}
In this subsection, we construct the sequence of Galerkin approximation.  
Let $\left( w_k\right)$ be an orthogonal basis of $H^1_0(\Omega)$ and 
an orthonormal basis of $L^2(\Omega)$. For each 
$m \in \N$, denote  $V_m :=\mbox{span}\{ w_1 , w_2 ,\dots, w_m \}$. For $0\leq t\leq T$, let
\begin{eqnarray}
u_m(t)
 :=
\sum_{k=1}^m C^k_m(t)w_k
\end{eqnarray}
satisfies
\begin{eqnarray}
C^k_m(0)
:=
(u_0, w_k).
\end{eqnarray}
Then $u_m \in V_m$ and $C^k_m(t)$ satisfies the following initial value problem given by 
\begin{eqnarray}
\left(C^k_m(t) \right)^{\prime}
+
\sum_{i,j=1}^d  
\left( B_{ij}(u_m)\frac{\partial u_m}{\partial x_j}, 
\frac{\partial w_k}{\partial x_i} \right) 
+
\sum_{i=1}^d \left( f_i^{\prime}(u_m)
\frac{\partial u_m}{\partial x_i} , w_k \right)
= 0
\label{IBVP-coeff}
\\
C^k_m(0) = (u_0, w_k)\label{IBVP-coeff_initialvalue}
\end{eqnarray}
for $k = 1, 2, \dots, m$. The symbol $(., .)$ denotes the inner product in $L^2(\Omega)$. 

\vspace{0.1cm}
Denote $C_m(t):= \left( C^1_m(t), C^2_m(t), \dots, C^m_m(t) \right)^T$ 
and $g:=\left( g_1, g_2, \dots, g_m \right)^T$. 
We get from \eqref{IBVP-coeff} that $\left(C_{m}\right)$ satisfies 
\begin{eqnarray}
\label{ODE-1}
C_m^{\prime} = g(C_m)
\end{eqnarray}
with 
\begin{eqnarray}
g_p(C_m)
=
-\sum_{i,j=1}^d  
\left(
 B_{ij} \Big( \sum_{k=1}^m C^k_m(t) w_k  \Big)
\sum_{k=1}^m C^k_m(t) \frac{\partial w_k}{\partial x_j}, 
\frac{\partial w_p}{\partial x_i} 
\right)
\nonumber
\\
-
\sum_{i=1}^d \left( f_i^{\prime} \Big(\sum_{k=1}^m C^k_m(t) w_k \Big) 
\sum_{k=1}^m C^k_m(t)\frac{\partial w_k}{\partial x_i}, w_p \right) 
\label{ODE}
\end{eqnarray} 
for $p = 1, 2, \dots, m$. 
 Then we have
\begin{eqnarray}
\norm{g_p(C_m)}
\le
\sum_{i,j=1}^d  
\norm{\left(
 B_{ij} \Big( \sum_{k=1}^m C^k_m(t) w_k  \Big)
\sum_{k=1}^m C^k_m(t) \frac{\partial w_k}{\partial x_j}, 
\frac{\partial w_p}{\partial x_i} 
\right)}
\nonumber
\\
+
\sum_{i=1}^d \norm{\left( f_i^{\prime} \Big(\sum_{k=1}^m C^k_m(t) w_k \Big) 
\sum_{k=1}^m C^k_m(t)\frac{\partial w_k}{\partial x_i}, w_p \right)}. 
\end{eqnarray}
from \eqref{ODE}. \text{Applying H\"{o}lder inequality, we obtain}
\begin{eqnarray}
\norm{g_p(C_m)}\le 
\displaystyle\sum_{i,j=1}^{d}\nrm{
 B_{ij} \Big( \sum_{k=1}^m C^k_m(t) w_k  \Big)
\sum_{k=1}^m C^k_m(t) \frac{\partial w_k}{\partial x_j}}_{L^2(\Omega)} 
\nrm{\frac{\partial w_p}{\partial x_i}}_{L^2(\Omega)}~~~~~~~~~~~~~~~
\nonumber
\\
+
\sum_{i=1}^d \nrm{ f_i^{\prime} \Big(\sum_{k=1}^m C^k_m(t) w_k \Big) 
\sum_{k=1}^m C^k_m(t)\frac{\partial w_k}{\partial x_i}}_{L^2(\Omega)} 
\nrm{w_p }_{L^2(\Omega)} 
\\
\le
\nrm{B}_{L^{\infty}(\R)}
\sum_{j=1}^d \sum_{k=1}^m
\norm{C^k_m(t)}\nrm{\frac{\partial w_k}{\partial x_j}}_{L^2(\Omega)}
\sum_{i=1}^d \nrm{\frac{\partial w_p}{\partial x_i}}_{L^2(\Omega)}~~~~~~~~~~~~~~~
\nonumber
\\
+
\nrm{f^{\prime}}_{\left(L^{\infty}(\R)\right)^d}\sum_{i=1}^d \sum_{k=1}^m
\norm{C^k_m(t)}\nrm{{\frac{\partial w_k}{\partial x_i}}}_{L^2(\Omega)}
\\
= 
\sum_{k=1}^m M^k_p \norm{C^k_m(t)}~~~~~~~~~~~~~~~~~~~~~~~~~~~~~~~~~~~~~~~~~~~~~~~~~~~~~~~~~~~~~~~~
\\
\le
\left(\sum_{k=1}^m (M^k_p)^2\right)^{\frac{1}{2}} 
\left(\sum_{k=1}^m{\norm{C^k_m(t)}}^2\right)^{\frac{1}{2}},~~~~~~~~~~~~~~~~~~~~~~~~~~~~~~~~~~~~~~~~~~
\end{eqnarray}
where
\begin{eqnarray*}
M^k_p
=
\nrm {B}_{L^{\infty}(\R)}\,\sum_{j=1}^d\nrm{\frac{\partial w_k}{\partial x_j}}_{L^2(\Omega)}
 \sum_{i=1}^d 
\nrm{\frac{\partial w_p}{\partial x_i}}_{L^2(\Omega)}
+
\nrm{f^{\prime}}_{\left(L^{\infty}(\R)\right)^{d}}\,\sum_{i=1}^d\nrm{\frac{\partial w_k}{\partial x_i}}_{L^2(\Omega)}.
\end{eqnarray*}
Therefore
\begin{eqnarray}
{\nrm{g(C_m)}}_2
\le
\alpha
{\nrm {C_m}}_2~,
\end{eqnarray} 
where 
\begin{eqnarray*}
\alpha 
=
\sqrt{\sum_{p=1}^m \sum_{k=1}^m (M^k_p)^2}
\end{eqnarray*}
and $\nrm ._2$ denote the Euclidean norm on $\R^m$.

Since $g$ is continuous, applying a result of ordinary
differential equation mentioned in \cite[p.73]{Vrabie}, we conclude that \eqref{IBVP-coeff} has a global solution. 
Therefore the sequence of Galerkin approximations $\left(u_m\right)$ is defined on the interval $[0,T]$. \\
\begin{remark}\label{Remark_galerkin_approximations}
{\rm Further, for $p=1,2,\cdots,m$, applications of Dominated Convergence theorem shows that 
\begin{eqnarray}
\frac{d}{dt}\Big(g_p(C_m)\Big)
=
- \sum_{i,j=1}^d \left( B_{ij}^{\prime}(u_m)\frac{\partial u_m}{\partial t}
\frac{\partial u_m}{\partial x_j}+
B_{ij}(u_m)\frac{\partial}{\partial x_j}\Big(\frac{\partial u_m}{\partial t}\Big),
\frac{\partial w_p}{\partial x_i}\right)
\nonumber
\\
-
\sum_{i=1}^d \left(\frac{d}{dt}\big(f_i^{\prime}(u_m)\big)
\frac{\partial u_m}{\partial t} \frac{\partial u_m}{\partial x_i}
+
f_i^{\prime}(u_m)\frac{\partial}{\partial x_i}
\Big( \frac{\partial u_m}{\partial t}\Big), w_p
\right).
\end{eqnarray}
Since each $C^k_m(t)$ is continuously differentiable in $[0,T]$, thus $\frac{d}{dt}\Big(g_p(C_m)\Big)$
exists and continuous in $[0,T]$. As a result, we have that  $C_m^{\prime\prime}(t)$ exists and continuous in $[0,T]$. Applying this arguments repeatedly,
we can conclude that $C_m(t)$ is a $C^{\infty}$-function for smooth $f$ and $\left(B_{ij}\right)$. }
\end{remark}
\subsection{Energy estimates}\label{Energy_Estimates}
In order to conclude Theorem \ref{Energy_Estimates_Theorem_1}, we follow the proof of Energy estimates from \cite[p.375]{Evans}. 
\begin{theorem}\label{Energy_Estimates_Theorem_1}
Let  $m\in\mathbb{N}$. There exists a constant $A > 0$ such that $\left(u_{m}\right)$ satisfies 
$$
\displaystyle\max_{0 \leq t \leq T} \nrm {u_m}_{L^2(\Omega)} 
+
\nrm {u_m}_{L^2(0,T;H^1_0(\Omega))} 
+ 
\nrm {\frac{\partial u_m}{\partial t}}_{L^2(0,T;H^{-1}(\Omega))}
\le 
A\,\,\nrm {u_0}_{L^2(\Omega)}.
$$

\end{theorem}

\begin{proof}
\textbf{Step-1.}
We know that the Galerkin approximations $\left(u_m\right)$ satisfies the equation
\begin{eqnarray}
\label{eq11}
\left(\frac{\partial u_m}{\partial t},w_k\right)
+
\sum_{i,j=1}^d \int_\Omega B_{ij}(u_m) 
\frac{\partial u_m}{\partial x_j} \frac{\partial w_k}{\partial x_i} dx
+
\sum_{i=1}^d \int_\Omega f_i^{\prime}(u_m)
\frac{\partial u_m}{\partial x_i}w_k dx
= 0.
\end{eqnarray}
Multiplying equation \eqref{eq11} by $C^k_m(t)$ 
and summing over $k = 1, 2, \dots, m$, we get
\begin{eqnarray}
\label {weak solution}
\left(\frac{\partial u_m}{\partial t}, u_m\right)
+
\sum_{i,j=1}^d \int_\Omega B_{ij}(u_m) 
\frac{\partial u_m}{\partial x_j} \frac{\partial u_m}{\partial x_i} dx
+
\sum_{i=1}^d \int_\Omega f_i^{\prime}(u_m) 
\frac{\partial u_m}{\partial x_i} u_m dx = 0.
\end{eqnarray}
Applying Hypothesis $(HB)$, we conclude that 
\begin{eqnarray}\label{Energy_Aditi_RM_eqn1}
\theta \int_\Omega |Du_m|^2\,dx 
~\le ~
\sum_{i,j=1}^d \int_{\Omega}\,B_{ij}(u_m) 
\frac{\partial u_m}{\partial x_j} \frac{\partial u_m}{\partial x_i}\,dx.
\end{eqnarray}
 There exists $\theta_{1}>0$ such that
\begin{eqnarray}
\theta_{1}\nrm {u_m}^2_{H^1_0(\Omega)} 
~\le~
\sum_{i,j=1}^d \int_\Omega B_{ij}(u_m) 
\frac{\partial u_m}{\partial x_j} \frac{\partial u_m}{\partial x_i} dx.
\end{eqnarray}
We further have 
\begin{eqnarray}
-\sum_{i=1}^d \Big( f_i^{\prime}(u_m) 
\frac{\partial u_m}{\partial x_i}, u_m \Big)
& \le & 
\frac{\epsilon}{2} \sum _{i=1}^d \int_{\Omega} 
\left| f_i^{\prime}(u_m) \frac{\partial u_m}{\partial x_i} \right|^2 dx
+
\frac{1}{2\epsilon} \sum_{i=1}^d \int_\Omega \left| u_m \right|^2 dx
\nonumber 
\\
& \le &
\frac{\epsilon}{2} \sum _{i=1}^d \operatorname*{ess\,sup}_{x \in\mathbb{R}} 
\left| f_i^{\prime}(x) \right|^2 \int_\Omega 
\left| \frac{\partial u_m}{\partial x_i} \right|^2 dx
+
\frac{d}{2\epsilon} \nrm {u_m}^2_{L^2(\Omega)} 
\nonumber
\\
& \le &
\frac{\epsilon}{2}  \displaystyle\max_{1 \leq i \leq d}\,\operatorname*{ess\,sup}_{x \in \R}
\left| f_i^{\prime}(x) \right|^2 \sum _{i=1}^d \int_{\Omega} 
\left| \frac{\partial u_m}{\partial x_i} \right|^2 dx
+
\frac{d}{2\epsilon}\nrm {u_m}^2_{L^2(\Omega)} 
\nonumber
\\
\!\!\!& \le \!\!\!&
\frac{\epsilon}{2} \nrm {f^{\prime}}_{(L^{\infty}(\R))^d} 
\sum _{i=1}^d \int_{\Omega} \left| \frac{\partial u_m}{\partial x_i} \right|^2 dx
+
\frac{d}{2\epsilon} \nrm {u_m}^2_{L^2(\Omega)}
\nonumber
\\
\!\!\!& = \!\!\!&
C \epsilon \nrm {u_m}^2_{H^1_0(\Omega)} 
+ 
\frac{D}{2\epsilon} \nrm {u_m}^2_{L^2(\Omega)}.
\end{eqnarray}
for $\epsilon>0$.
\\
We know that 
$
(\frac{\partial u_m}{\partial t},u_m) = \frac{d}{dt} 
(\frac{1}{2} \nrm {u_m}^2_{L^2(\Omega)})
$
for almost every $ 0 \le t \le T $.
We obtain   
\begin{eqnarray}
\nonumber
\frac{d}{dt} \nrm {u_m}^2_{L^2(\Omega)}
+2\theta_{1} \nrm {u_m}^2_{H^1_0(\Omega)}
 \le 
2C\epsilon \nrm {u_m}^2_{H^1_0(\Omega)}
+
\frac{D}{\epsilon} \nrm {u_m}^2_{L^2(\Omega)}.
\\
\frac{d}{dt} \nrm {u_m}^2_{L^2(\Omega)}
+2(\theta_{1}-C\epsilon) \nrm {u_m}^2_{H^1_0(\Omega)}
\le
\frac{D}{\epsilon} \nrm {u_m}^2_{L^2(\Omega)}
\end{eqnarray}
for almost every $0 \le t \le T$ from equation \eqref{weak solution} by combining all the above results. We choose $\epsilon $ such that $\theta_{1}-C\epsilon > 0$. Therefore, there exists a constant $C_{1}>0$ such that 
\begin{eqnarray}
\label{energy estimate-1}
\frac{d}{dt} \nrm {u_m}^2_{L^2(\Omega)}
+C_{1} \nrm {u_m}^2_{H^1_0(\Omega)}
\!\!\!& \le \!\!\!&
\frac{D}{\epsilon} \nrm {u_m}^2_{L^2(\Omega)}.
\end{eqnarray}

\textbf{Step-2.}
Let 
\begin{eqnarray}
\eta(t) = \nrm {u_m}^2_{L^2(\Omega)}.
\end{eqnarray}
Then \eqref{energy estimate-1} implies 
\begin{equation}
\eta^{\prime}(t) \le \frac{D}{\epsilon}\eta(t)
\end{equation}
for almost every $0\le t \le T$. Thus using differential form of 
Gronwall's inequality yields the estimate
\\
\begin{eqnarray}
\eta(t) \le e^{D_1t} \eta(0). 
\end{eqnarray}
\\
Since $\eta(0)
=
\nrms{u_m(0)}_{L^2(\Omega)} 
\le
\nrms{u_0}_{L^2(\Omega)}$, there exists a constant $C_{2}>0$ such that  the estimate
\begin{equation*}
\displaystyle\max_{0\le t \le T} \nrm {u_m}^2_{L^2(\Omega)} 
\le 
C_{2} \nrm{u_0}^2_{L^2(\Omega)}
\end{equation*}
holds.\\
\vspace{0.1cm}\\
\textbf{Step-3.}
Using\,\,$\displaystyle\max_{0\le t \le T} \nrm {u_m}^2_{L^2(\Omega)} 
\le 
C \nrm{u_0}^2_{L^2(\Omega)}$\,\, in the following computation, we conclude that there exists a constant $C_{3}>0$ such that 
\begin{eqnarray*}
\nrm {u_m}^2_{L^2(0,T,H^1_0(\Omega))} 
\!\!\!& = \!\!\!& 
\int_0^T \nrm {u_m}^2_{H^1_0(\Omega)}
\\
\!\!\!& \le \!\!\!&
C_{3}\nrm{u_0}^2_{L^2(\Omega)}.
\end{eqnarray*}
\vspace{0.1cm}\\
\textbf{Step-4.}
Let $v \in H^1_0(\Omega)$ such that $\nrm{v}_{H^1_0(\Omega)} \le 1$. Let $v=v^1+v^2$ such that $v^1\in \mbox{span}\left( w_k \right)_{k=1}^m$ 
and $(v^2,w_k)=0~\text{for}~k = 1,2,\dots,m.\,\text{Then} ~(v, w_k)=(v^1,w_k)$. Therefore 
$
\nrm{v^1}_{H^1_0(\Omega)}
\le 
\nrm{v}_{H^1_0(\Omega)}
\le
1.
$
For almost every $0\le t \le T$, utilizing \eqref{eq11}, we deduce that
\begin{eqnarray*}
\left(\frac{\partial u_m}{\partial t},v^1\right)
+
\sum_{i,j=1}^d \Big( B_{ij}(u_m)
\frac{\partial u_m}{\partial x_j}, 
\frac{\partial v^1}{\partial x_i} \Big)
+
\sum_{i=1}^d \Big(f_i^{\prime}(u_m)
\frac{\partial u_m}{\partial x_i}, v^1\Big)
=
0.
\end{eqnarray*}
Observe that $\inp{\frac{\partial u_m}{\partial t}}v
~=~(\frac{\partial u_m}{\partial t}, v)
~=~(\frac{\partial u_m}{\partial t}, v^1).$ 
Then 
\begin{eqnarray}\label{POIAOSC_Change_Eqn1}
\left|\inp{\frac{\partial u_m}{\partial t}}v\right|
\!\!\!& \le \!\!\!&
\sum_{i,j=1}^d \left|\Big( B_{ij}(u_m)
\frac{\partial u_m}{\partial x_j}, 
\frac{\partial v^1}{\partial x_i} \Big)\right|
+
\sum_{i=1}^d \left|\Big(f_i^{\prime}(u_m)
\frac{\partial u_m}{\partial x_i}, v^1\Big)\right|\nonumber
\\
\!\!\!& \le \!\!\!&
\sum_{i,j=1}^d \nrm{B_{ij}(u_m)
\frac{\partial u_m}{\partial x_j}}_{L^2(\Omega)} 
\nrm{\frac{\partial v^1}{\partial x_i}}_{L^2(\Omega)}
+
\sum_{i=1}^d \nrm{f^{\prime}_i(u_m)
\frac{\partial u_m}{\partial x_i}}_{L^2(\Omega)}
\nrm{v^1}_{L^2(\Omega)}\nonumber
\\
\!\!\!& \le \!\!\!&
C_{4}\sum_{i=1}^d 
\nrm{
\frac{\partial u_m}{\partial x_j}
}_{L^2(\Omega)} 
\nrm{v^1}_{H^1_0(\Omega)} 
+
D_{2}\nrm {u_m}_{H^1_0(\Omega)}
\nrm{v^1}_{L^2(\Omega)}\nonumber
\\
& \le &
E\nrm{u_m}_{H^1_0(\Omega)}. 
\end{eqnarray}
Thus,
\begin{eqnarray*}
\nrm{\frac{\partial u_m}{\partial t}}_{H^{-1}(\Omega)}
\le
E\nrm {u_m}_{H^1_0(\Omega)}.
\end{eqnarray*}
Integrating over $[0,T]$ and using the inequality proved in \textbf{Step-3}, we get the existence of a constant $E_{1}>0$ such that
\begin{eqnarray*}
\nrm{\frac{\partial u_m}{\partial t}}_{L^2(0,T,H^{-1}(\Omega))}
\le
E_{1}\,\nrm{u_0}_{L^2(\Omega)}.
\end{eqnarray*}
Hence, we have the existence of a constant $A>0$ such that
\begin{eqnarray}
\displaystyle\max_{0\le t\le T}\nrm {u_m}_{L^2(\Omega)}
+
\nrm {u_m}_{L^2(0,T;H^1_0(\Omega))}
+
\nrm {\frac{\partial u_m}{\partial t}}_{L^2(0,T;H^{-1}(\Omega))}
\le
A\,\nrm{u_0}_{L^2(\Omega)}.
\end{eqnarray}
\end{proof}
\subsection{BV-Estimate}\label{BV_Estimates}
In this section, we prove BV-estimates of Galerkin approximations $\left(u_m\right)$, 
{\it i.e}, we want to show that for $m\in\mathbb{N}$, there exists a constant $C>0$ such that the following inequality
\begin{eqnarray}
\nrm{\frac{\partial u_m}{\partial t}}_{L^1(\Omega_T)}
+
\nrm {\nabla u_m}_{\left(L^1(\Omega _T)\right)^d}
\le
C
\end{eqnarray}
holds. \\For this, we first show that the sequence of Galerkin approximations 
$\left(u_m\right)$ satisfies the equation \eqref{eq1.1}.
We recall the following equation    
\begin{eqnarray}
\label{weak u_m 2}
\left(\frac{\partial u_m}{\partial t},w_k \right)
+
\sum_{i,j=1}^d \int_\Omega B_{ij}(u_m) 
\frac{\partial u_m}{\partial x_j} \frac{\partial w_k}{\partial x_i} dx
+
\sum_{i=1}^d \int_\Omega f_i^{\prime}(u_m)
\frac{\partial u_m}{\partial x_i}w_k dx
= 0
\end{eqnarray} 
satisfied by $u_m$. Applying integration by parts to the second term on LHS of \eqref{weak u_m 2}, we get
\begin{eqnarray}
\left(\frac{\partial u_m}{\partial t},w_k \right)
-
\sum_{i,j=1}^d \int_\Omega \frac{\partial}{\partial x_i}\Big(B_{ij}(u_m) 
\frac{\partial u_m}{\partial x_j}\Big)  w_k dx
+
\sum_{i=1}^d \int_\Omega f_i^{\prime}(u_m)
\frac{\partial u_m}{\partial x_i}w_k dx
=
0.
\end{eqnarray} 
Multiplying the above equation with $d_k$ and 
summing over 1 to N, we have
\begin{eqnarray}
\left(\frac{\partial u_m}{\partial t},\sum_{k=1}^N d_kw_k \right)
-
\sum_{i,j=1}^d \int_\Omega \frac{\partial}{\partial x_i}\Big(B_{ij}(u_m) 
\frac{\partial u_m}{\partial x_j}\Big) \sum_{k=1}^N d_kw_k dx
\nonumber
\\
+
\sum_{i=1}^d \int_\Omega f_i^{\prime}(u_m)
\frac{\partial u_m}{\partial x_i} \sum_{k=1}^N d_kw_k dx
=
0.
\end{eqnarray} 
Denote $S_N:=\sum_{k=1}^N d_kw_k$. Then the above equation becomes
\begin{eqnarray}
\label{inr prdt u_m 1}
\left(\frac{\partial u_m}{\partial t} 
-
\sum_{i,j=1}^d  \frac{\partial}{\partial x_i}\Big(B_{ij}(u_m) 
\frac{\partial u_m}{\partial x_j}\Big) 
+
\sum_{i=1}^d  f_i^{\prime}(u_m)
\frac{\partial u_m}{\partial x_i}, S_N \right)
=
0.
\end{eqnarray}
It is clear that $S_N$ represents the partial sum of the series 
$v=\sum_{k=1}^\infty d_kw_k$. We now show that 
\eqref{inr prdt u_m 1} holds for all $v \in L^2(\Omega)$. 
Replacing $S_N$ by $S_N-v$ in \eqref{inr prdt u_m 1} and 
applying H\"{o}lder inequality, we get the following inequality  
\begin{eqnarray}
\norm{
\left(\frac{\partial u_m}{\partial t} 
-
\sum_{i,j=1}^d  \frac{\partial}{\partial x_i}\Big(B_{ij}(u_m) 
\frac{\partial u_m}{\partial x_j}\Big) 
+
\sum_{i=1}^d  f_i^{\prime}(u_m)
\frac{\partial u_m}{\partial x_i}, S_N-v \right)
}
\nonumber
\le
\\
\label{inr prdt u_m 2}
\nrm{
\frac{\partial u_m}{\partial t} 
-
\sum_{i,j=1}^d  \frac{\partial}{\partial x_i}\Big(B_{ij}(u_m) 
\frac{\partial u_m}{\partial x_j}\Big) 
+
\sum_{i=1}^d  f_i^{\prime}(u_m)
\frac{\partial u_m}{\partial x_i}
}_{L^2(\Omega)}
\nrm{S_N-v}_{L^2(\Omega)}.
\end{eqnarray}
Since $S_N$ converges to $v$ in $L^2(\Omega)$, 
left hand side of the inequality of \eqref{inr prdt u_m 2} 
converges to 0 in $L^2(\Omega)$ i.e
\begin{eqnarray}
\lim_{N \to \infty}
\left(\frac{\partial u_m}{\partial t} 
-
\sum_{i,j=1}^d  \frac{\partial}{\partial x_i}\Big(B_{ij}(u_m) 
\frac{\partial u_m}{\partial x_j}\Big) 
+
\sum_{i=1}^d  f_i^{\prime}(u_m)
\frac{\partial u_m}{\partial x_i}, S_N \right)
\nonumber
=
\\
\left(\frac{\partial u_m}{\partial t} 
-
\sum_{i,j=1}^d  \frac{\partial}{\partial x_i}\Big(B_{ij}(u_m) 
\frac{\partial u_m}{\partial x_j}\Big) 
+
\sum_{i=1}^d  f_i^{\prime}(u_m)
\frac{\partial u_m}{\partial x_i}, v \right).
\end{eqnarray}
Using \eqref{inr prdt u_m 1}, we obtain
\begin{eqnarray}\label{Aditi_RM_GalerkinEquation.1}
\left(\frac{\partial u_m}{\partial t} 
-
\sum_{i,j=1}^d  \frac{\partial}{\partial x_i}\Big(B_{ij}(u_m) 
\frac{\partial u_m}{\partial x_j}\Big) 
+
\sum_{i=1}^d  f_i^{\prime}(u_m)
\frac{\partial u_m}{\partial x_i}, v \right)
=
0
\end{eqnarray}
for all $v \in L^2(\Omega)$.
Then we have that
\begin{eqnarray}
\frac{\partial u_m}{\partial t} 
-
\sum_{i,j=1}^d  \frac{\partial}{\partial x_i}\Big(B_{ij}(u_m) 
\frac{\partial u_m}{\partial x_j}\Big) 
+
\sum_{i=1}^d  f_i^{\prime}(u_m)
\frac{\partial u_m}{\partial x_i} 
=
0
\end{eqnarray}
 in $L^2(\Omega)$. Since left hand side of the above equation is continuous in $\Omega_{T}$, we obtain that
\begin{eqnarray}
\label{eqn-u_m}
\frac{\partial u_m}{\partial t} 
-
\sum_{i,j=1}^d  \frac{\partial}{\partial x_i}\Big(B_{ij}(u_m) 
\frac{\partial u_m}{\partial x_j}\Big) 
+
\sum_{i=1}^d  f_i^{\prime}(u_m)
\frac{\partial u_m}{\partial x_i} 
=
0
\end{eqnarray}
for all $\left(x,t\right)$ in $\Omega_{T}$.   
Therefore, we conclude that $u_m$ satisfies the equation \eqref{eq1.1}.
\\
\vspace{0.1cm}\\
We now state the theorem concerning BV-estimate formally.
\begin{theorem}
\label{BV}
 Let f, B, $u_0$ satisfy the \textbf{Hypothesis-H.} 
Then, for all $m\in\mathbb{N}$, there exists a constant $C> 0$ 
such that the sequence of Galerkin approximations $\left(u_{m}\right)$ satisfies the following estimate:
\begin{eqnarray}\label{BV_Aditi_RM_formal_equation_1}
\nrm{\frac{\partial u_{m}}{\partial t}}_{L^1(\Omega_T)}
+
\nrm {\nabla u_{m}}_{\left(L^1(\Omega _T)\right)^{d}}
\le
C.
\end{eqnarray}
\end{theorem}  
To prove this inequality \eqref{BV_Aditi_RM_formal_equation_1}, we use the following sequence of smooth approximations
$$
\mbox{sg}_n(s):=\mbox{tanh}\,\,ns
$$ 
for the signum function
$$
\mbox{sg}(s) = 
\begin{cases}
1 & \text{if }\phantom{m} s> 0\\
0 & \text{if }\phantom{m} s= 0\\
-1 & \text{if }\phantom{m} s< 0.
\end{cases}
$$
 The multiplier $\left(\mbox{sg}_n(s)\right)$ is used 
in \cite{Ramesh} to prove the BV-estimate of IBVP for quasilinear viscous approximations, where the viscous term is of the form
$\varepsilon\nabla.\Big(B(u)\nabla u \Big)$ of \eqref{eq1.1}. Observe that for $s \in \R$,  
$$
\mbox{sg}_n^{\prime}(s) = n\,\mbox{sech}^2\,\,ns.
$$
Therefore
\begin{eqnarray}
\label{signum-1}
\norm{s}\,\mbox{sg}_n^{\prime}(s)\le 4 \phantom{m}
\forall \phantom{m} s\in \{\R\}
\end{eqnarray} 
and for all $s \in \R$, we have
\begin{eqnarray}
\label{signum-2}
\norm{s}\mbox{sg}_n^{\prime}(s) \to 0
~\text{as}~ n \to \infty.
\end{eqnarray}
\textbf{Proof of Theorem \ref{BV}:}\\ 
We prove Theorem \ref{BV} in three steps. In \textbf{Step-1}, we prove the $L^{1}\left(\Omega_{T}\right)$ estimates of $\left(\frac{\partial u_{m}}{\partial t}\right)$. In \textbf{Step-2}, we prove an estimate $\left\|\nabla u_{m}\right\|_{\left(L^{1}\left(\Omega_{T}\right)\right)^{d}}$ and in \textbf{Step-3}, we conclude the proof of Theorem \ref{BV} combining \textbf{Step-1} and \textbf{Step-2}.\\
\vspace{0.1cm}\\
\textbf{Step-1:}
We first prove that there exists a constant $C_1 > 0$ such that
\begin{eqnarray}
\nrm{\frac{\partial u_m}{\partial t}}_{L^1(\Omega_T) }\le C_1.
\end{eqnarray}
As we know that $u_m$ satisfies the equation \eqref{eq1.1}  i.e
\begin{eqnarray*}
\frac{\partial u_m}{\partial t} 
+
\sum_{i=1}^d \frac{\partial}{\partial x_i}f_i(u_m)
=
\sum_{i,j=1}^d\frac{\partial}{\partial x_i}
\left( B_{ij}(u_m)\frac{\partial u_m}{\partial x_j} \right).
\end{eqnarray*}
In view of \textbf{Remark \ref{Remark_galerkin_approximations}}, we differentiate the above equation with respect to t to obtain
\begin{eqnarray}
\frac{\partial^2 u_m}{\partial t^2}
+
\sum_{i=1}^d \frac{\partial}{\partial x_i}
\left(f_i^{\prime}(u_m)\frac{\partial u_m}{\partial t}\right)
=
\sum_{i,j=1}^d \frac{\partial}{\partial x_i}
\left( B_{ij}^{\prime}(u_m)\frac{\partial u_m}{\partial t}
\frac{\partial u_m}{\partial x_j}\right)
\nonumber
\\
+
\sum_{i,j=1}^d \frac{\partial}{\partial x_i}
\left(B_{ij}(u_m)\frac{\partial}{\partial x_j} 
\Big(\frac{\partial u_m}{\partial t}\Big)\right).
\end{eqnarray}
Multiplying with $\mbox{sg}_n(\frac{\partial u_m}{\partial t})$ and using 
integrating by parts over $\Omega$, we have
\begin{eqnarray}
\label{BV-sg eqn}
\int_\Omega \frac{\partial^2 u_m}{\partial t^2}\,
\mbox{sg}_n\Big( \frac{\partial u_m}{\partial t}\Big)dx
=
\sum_{i=1}^d \int_\Omega
f_i^{\prime}(u_m)\frac{\partial u_m}{\partial t}
\frac{\partial}{\partial x_i}
\left( \mbox{sg}_n\Big( \frac{\partial u_m}{\partial t} \Big)\right) dx~~~~~~~~~~~~~~~~~~~~~~~~~~~
\nonumber
\\
- 
\sum_{i,j=1}^d \int_\Omega 
B_{ij}^{\prime}(u_m)\frac{\partial u_m}{\partial t}
\frac{\partial u_m}{\partial x_j}
\frac{\partial}{\partial x_i}
\left(\mbox{sg}_n\Big( \frac{\partial u_m}{\partial t} \Big)\right) dx~~~~~~~~~~~~~~~
\nonumber
\\
-
\sum_{i,j=1}^d \int_\Omega
B_{ij}(u_m)\frac{\partial}{\partial x_j} 
\Big(\frac{\partial u_m}{\partial t}\Big)
\frac{\partial}{\partial x_i}
\left(\mbox{sg}_n\Big( \frac{\partial u_m}{\partial t} \Big)\right) dx.~~~
\end{eqnarray}
We now prove
\begin{eqnarray}
\label{BV-lim-1}
\lim_{n \to \infty}
\sum_{i=1}^d \int_\Omega
f_i^{\prime}(u_m)\frac{\partial u_m}{\partial t}
\frac{\partial}{\partial x_i}
\left( \mbox{sg}_n\Big( \frac{\partial u_m}{\partial t} \Big)\right) dx
= 0 
\end{eqnarray} 
and
\begin{eqnarray}
\label{BV-lim-2}
\lim_{n \to \infty}
\sum_{i,j=1}^d \int_\Omega 
B_{ij}^{\prime}(u_m)\frac{\partial u_m}{\partial t}
\frac{\partial u_m}{\partial x_j}
\frac{\partial}{\partial x_i}
\left(\mbox{sg}_n\Big( \frac{\partial u_m}{\partial t} \Big)\right) dx
=
0.
\end{eqnarray}
Using the fact \eqref{signum-1}, we obtain
\begin{eqnarray}
\sum_{i=1}^d\norm{f_i^{\prime}(u_m)
\frac{\partial u_m}{\partial t}\,\mbox{sg}_n^{\prime}
\Big(\frac{\partial u_m}{\partial t}\Big)
\frac{\partial}{\partial x_i}
\Big(\frac{\partial u_m}{\partial t}\Big) dx}~~~~~~~~~~~~~~~~~~~~~~~~~~~
\nonumber
\\
\le
\sum_{i=1}^d
\norm{f_i^{\prime}(u_m)}\norm{\frac{\partial u_m}{\partial t}}
\,\mbox{sg}_n^{\prime}\Big(\frac{\partial u_m}{\partial t}\Big)
\norm{\frac{\partial}{\partial x_i}
\Big( \frac{\partial u_m}{\partial t}\Big)}dx
\nonumber
\\
\le
4\nrm{f^{\prime}}_{(L^{\infty}(\R))^d}
\sum_{i=1}^d \norm{\frac{\partial}{\partial x_i}
\Big( \frac{\partial u_m}{\partial t}\Big)}dx.
\end{eqnarray}
The right hand side of the above inequality is in $L^1(\Omega)$ as it is in $L^2(\Omega)$ and $\Omega$ is bounded. Therefore using \eqref{signum-2} 
and an application of Dominated convergence theorem gives \eqref{BV-lim-1}.
\\
Similarly 
\begin{eqnarray}
\sum_{i,j=1}^d\norm{ 
B_{ij}^{\prime}(u_m)\frac{\partial u_m}{\partial t}\,
\mbox{sg}_n^{\prime}\Big(\frac{\partial u_m}{\partial t}\Big)
\frac{\partial u_m}{\partial x_j}\frac{\partial}{\partial x_i}
\Big(\frac{\partial u_m}{\partial t} \Big) dx}~~~~~~~~~~~~~~~~~~~~~~~~~~~~~~~
\nonumber
\\
\le
4\nrm{B}_{L^{\infty}(\R)}\displaystyle\max_{1\le j \le d}
\left( \nrm{\frac{\partial u_m}{\partial x_j}}_
{L^{\infty}(\Omega_T)}\right)
\sum_{i=1}^d\norm{\frac{\partial}{\partial x_i}
\Big(\frac{\partial u_m}{\partial t}\Big)}dx.
\end{eqnarray}
We conclude \eqref{BV-lim-2} using a similar argument.

We now consider the following third term
\begin{eqnarray} 
\sum_{i,j=1}^d \!\!\!\!\!\!\!\!\!\!\!\!\!&&\int_\Omega
 B_{ij}(u_m)\frac{\partial}{\partial x_j} 
\Big(\frac{\partial u_m}{\partial t}\Big)
\frac{\partial}{\partial x_i}
\left(\mbox{sg}_n\Big( \frac{\partial u_m}{\partial t} \Big)\right) dx
\nonumber
\\
& = &
\sum_{i,j=1}^d \int_\Omega
B_{ij}(u_m)\frac{\partial}{\partial x_j} 
\Big(\frac{\partial u_m}{\partial t}\Big)
\frac{\partial}{\partial x_i}\Big( \frac{\partial u_m}{\partial t} \Big)
\left(\mbox{sg}_n^{\prime}\Big( \frac{\partial u_m}{\partial t} \Big)\right) dx.
\end{eqnarray}
Since $\sum_{i,j=1}^d B_{ij} \xi_i \xi_j \ge 0
~\text{ for all }~\xi \in \R^d$ and 
$\mbox{sg}_n^{\prime}(s)>0~\text{for all}~s\in \R$, 
we can have
\begin{eqnarray}
\sum_{i,j=1}^d \int_\Omega
B_{ij}(u_m)\frac{\partial}{\partial x_j} 
\Big(\frac{\partial u_m}{\partial t}\Big)
\frac{\partial}{\partial x_i}\Big( \frac{\partial u_m}{\partial t} \Big)
\left(\mbox{sg}_n^{\prime}\Big( \frac{\partial u_m}{\partial t} \Big)\right) dx
\ge 
0.
\end{eqnarray}
Therefore taking limit supremum on both sides of \eqref{BV-sg eqn}, we have
\begin{eqnarray}
\limsup_{n \to \infty}
\int_\Omega \frac{\partial^2 u_m}{\partial t^2}
\,\mbox{sg}_n\Big( \frac{\partial u_m}{\partial t}\Big)dx
\le
0.
\end{eqnarray}
Since $\norm{\frac{\partial^2 u_m}{\partial t^2}\,
\mbox{sg}_n\Big( \frac{\partial u_m}{\partial t}\Big)}
\le
\norm{\frac{\partial^2 u_m}{\partial t^2}},
$
an application of Dominated convergence theorem shows that
\begin{eqnarray}
\lim_{n\to \infty} 
\int_\Omega \frac{\partial^2 u_m}{\partial t^2}
\mbox{sg}_n\Big( \frac{\partial u_m}{\partial t}\Big)dx
=
\int_\Omega \frac{\partial^2 u_m}{\partial t^2}
\,\mbox{sg}\Big( \frac{\partial u_m}{\partial t}\Big)dx.
\end{eqnarray}
Therefore
\begin{eqnarray}
\frac{d}{dt}\int_{\Omega}\norm{\frac{\partial u_m}{\partial t}}dx
=
\int_\Omega \frac{\partial^2 u_m}{\partial t^2}
\,\mbox{sg}\Big( \frac{\partial u_m}{\partial t}\Big)dx
\le 
0.
\end{eqnarray}
Then $\int_{\Omega}\norm{\frac{\partial u_m}{\partial t}}dx$
is a decreasing function with respect to t in the interval $[0,T]$.
\\
Thus, for all $0<t<T$, we have
\begin{eqnarray}\label{Aditi_RM_BV_Equations_12}
\int_{\Omega}\norm{\frac{\partial u_m}{\partial t}(x,t)}dx
\le
\int_{\Omega}\norm{\frac{\partial u_m}{\partial t}(x,0)}dx.
\end{eqnarray} 
Then from \eqref{eqn-u_m}, we have 
\begin{eqnarray}\label{Aditi_rm_BV_equation2}
\frac{\partial u_m}{\partial t}(x,0)
=
\sum_{i,j=1}^d B_{ij}^{\prime}(u_m(x,0))
\frac{\partial u_m}{\partial x_i}(x,0)
\frac{\partial u_m}{\partial x_j}(x,0)
\nonumber
\\
+
\sum_{i,j=1}^d B_{ij}(u_m(x,0))
\frac{\partial^2 u_m}{\partial x_i\partial x_j}(x,0)
\nonumber
\\
-
\sum_{i=1}^d f_i^{\prime}(u_m(x,0))
\frac{\partial u_m}{\partial x_i}(x,0).
\end{eqnarray}
\\
Therefore, we have
\begin{equation}\label{Aditi_rm_BV_equation4}
\begin{split}
\int_{\Omega}\norm{\frac{\partial u_m}{\partial t}(x,0)}dx \le \sum_{i,j=1}^d \int_{\Omega}\,\left|B_{ij}^{\prime}(u_m(x,0))\,
\frac{\partial u_m}{\partial x_i}(x,0)
\frac{\partial u_m}{\partial x_j}(x,0)\right|\,\,dx\\
+ \sum_{i,j=1}^d\int_{\Omega} \Big|B_{ij}(u_m(x,0))\,
\frac{\partial^2 u_m}{\partial x_i\partial x_j}(x,0)\Big|\,dx\\
+\sum_{i=1}^d\,\int_{\Omega}\,\,\left|f_i^{\prime}(u_m(x,0))
\frac{\partial u_m}{\partial x_i}(x,0)\right|\,dx.
\end{split}
\end{equation}
Applying H\"{o}lder inequality, we get
\begin{equation}\label{Aditi_rm_BV_equation5}
\begin{split}
\int_{\Omega}\norm{\frac{\partial u_m}{\partial t}(x,0)}dx\hspace{10 cm}\\\leq \displaystyle\max_{1\leq i,j\leq d}\|B_{i,j}^{\prime}\|_{L^{\infty}\left(\mathbb{R}\right)}\sum_{i,j=1}^d \displaystyle\max_{1\leq i,j\leq d}
\left\|\frac{\partial u_m}{\partial x_i}(x,0)\right\|_{L^{2}\left(\Omega\right)}
\left\|\frac{\partial u_m}{\partial x_j}(x,0)\right\|_{L^{2}\left(\Omega\right)}\\
+ \displaystyle\max_{1\leq i,j\leq d}\|B_{ij}\|_{L^{\infty}\left(\mathbb{R}\right)}\,\,\left(\mbox{Vol}\left(\Omega\right)\right)^{\frac{1}{2}}\,\sum_{i,j=1}^{d}\,\,\,
\left\|\frac{\partial^2 u_m}{\partial x_i\partial x_j}(x,0)\right\|_{L^{2}\left(\Omega\right)}\\
+\displaystyle\max_{1\leq i\leq d}\|f_{i}^{\prime}\|_{L^{\infty}\left(\mathbb{R}\right)}\sum_{i=1}^d\,
\left\|\frac{\partial u_m}{\partial x_i}(x,0)\right\|_{L^{2}\left(\Omega\right)}\,\left(\mbox{Vol}\left(\Omega\right)\right)^{\frac{1}{2}}.\\
\leq \displaystyle\max_{1\leq i,j\leq d}\|B_{i,j}^{\prime}\|_{L^{\infty}\left(\mathbb{R}\right)}\sum_{i,j=1}^d \displaystyle\max_{1\leq i,j\leq d}
\left\|\frac{\partial u_m}{\partial x_i}(x,0)\right\|_{L^{2}\left(\Omega\right)}
\left\|\frac{\partial u_m}{\partial x_j}(x,0)\right\|_{L^{2}\left(\Omega\right)}\\
+ \displaystyle\max_{1\leq i,j\leq d}\|B_{ij}\|_{L^{\infty}\left(\mathbb{R}\right)}\,\,\left(\mbox{Vol}\left(\Omega\right)\right)^{\frac{1}{2}}\,d\left(\sum_{i,j=1}^{d}\,\,\,
\left\|\frac{\partial^2 u_m}{\partial x_i\partial x_j}(x,0)\right\|^{2}_{L^{2}\left(\Omega\right)}\right)^{\frac{1}{2}}\\
+\displaystyle\max_{1\leq i\leq d}\|f_{i}^{\prime}\|_{L^{\infty}\left(\mathbb{R}\right)}\sum_{i=1}^d\,
\left\|\frac{\partial u_m}{\partial x_i}(x,0)\right\|_{L^{2}\left(\Omega\right)}\,\left(\mbox{Vol}\left(\Omega\right)\right)^{\frac{1}{2}}.\\
\leq \displaystyle\max_{1\leq i,j\leq d}\|B_{i,j}^{\prime}\|_{L^{\infty}\left(\mathbb{R}\right)}\sum_{i,j=1}^d \displaystyle\max_{1\leq i,j\leq d}
\left\|\frac{\partial u_m}{\partial x_i}(x,0)\right\|_{L^{2}\left(\Omega\right)}
\left\|\frac{\partial u_m}{\partial x_j}(x,0)\right\|_{L^{2}\left(\Omega\right)}\\
+ \displaystyle\max_{1\leq i,j\leq d}\|B_{ij}\|_{L^{\infty}\left(\mathbb{R}\right)}\,\,\left(\mbox{Vol}\left(\Omega\right)\right)^{\frac{1}{2}}\,d\,\,\,\|u_{m}(x,0)\|_{H^{2}\left(\Omega\right)}\\
+\displaystyle\max_{1\leq i\leq d}\|f_{i}^{\prime}\|_{L^{\infty}\left(\mathbb{R}\right)}\sum_{i=1}^d\,
\left\|\frac{\partial u_m}{\partial x_i}(x,0)\right\|_{L^{2}\left(\Omega\right)}\,\left(\mbox{Vol}\left(\Omega\right)\right)^{\frac{1}{2}}.
\end{split}
\end{equation}
There exists a constant $E>0$ from \cite[p.384-p.385]{Evans} such that $\left(u_{m}(0)\right)$ satisfies
\begin{equation}\label{Aditi_RM_Mistake_equation_2}
\|u_{m}(.,0)\|_{H^{2}\left(\Omega\right)}\leq E\|u_{0}\|_{H^{2}\left(\Omega\right)}.
\end{equation}
for $m\in\mathbb{N}$. \\
We now reproduce the proof of \eqref{Aditi_RM_Mistake_equation_2} from \cite[p.384-p.385]{Evans} for the completeness of the paper. Since $\left(w_{k}\right)$ satisfies $-\Delta w_{k}=\lambda_{k}\,w_{k}$ in $\Omega$ and $w_{k}=0$ on $\partial\Omega$, therefore $\Delta u_{m}=0$ on $\partial\Omega$.\\
We observe from \cite[p.152]{Kesavan} that $\|\Delta u_{m}(x,0)\|_{L^{2}\left(\Omega\right)}$ defines a norm on $H^{2}\left(\Omega\right)\cap H^{1}_{0}\left(\Omega\right)$ which is equivalent to the norm $\|u_{m}(x,0)\|_{H^{2}\left(\Omega\right)}$. Therefore there exists a constant $C_{1}>0$ such that 
\begin{equation}\label{Aditi_RM_Mistake_equation_3}
\begin{split}
\|u_{m}(x,0)\|^{2}_{H^{2}\left(\Omega\right)}\leq C_{1}\,\|\Delta u_{m}(x,0)\|^{2}_{L^{2}\left(\Omega\right)}=C_{1}\left(\Delta u_{m}(x,0), \Delta u_{m}(x,0)\right)\hspace{3cm}\\
=C_{1}\displaystyle\sum_{j=1}^{d}\left(\frac{\partial^{2}}{\partial x_{j}^{2}} u_{m}(x,0), \Delta u_{m}(x,0)\right)=-C_{1}\displaystyle\sum_{j=1}^{d}\left(\frac{\partial}{\partial x_{j}} u_{m}(x,0), \frac{\partial}{\partial x_{j}}\left(\Delta u_{m}(x,0)\right)\right)\\
=C_{1}\displaystyle\sum_{j=1}^{d}\left( u_{m}(x,0), \frac{\partial^{2}}{\partial x_{j}^{2}}\left(\Delta u_{m}(x,0)\right)\right)= C_{1}\left( u_{m}(x,0),\Delta\left(\Delta u_{m}(x,0)\right)\right)\\
=C_{1}\left( u_{m}(x,0),\Delta^{2} u_{m}(x,0)\right).
\end{split}	
\end{equation}
For $k=1,2,\cdots,m$, $u_{m}(x,0)$ satisfies
$$\left(u_{m}(x,0),w_{k}\right)=\left(u_{0},w_{k}\right).$$
Therefore, we obtain 
$$\left(u_{m}(x,0)-u_{0},w_{k}\right)=0$$
 and hence $u_{m}\left(x,0\right)=u_{0}$ in $\mbox{span}\left\{w_{1},w_{2},\cdots,w_{m}\right\}$. Since $\Delta^{2} u_{m}(x,0)\in\,\mbox{span}\left\{w_{1},w_{2},\cdots,w_{m}\right\}$ and applying $u_{m}\left(x,0\right)=u_{0}$ in $\mbox{span}\left\{w_{1},w_{2},\cdots,w_{m}\right\}$ in \eqref{Aditi_RM_Mistake_equation_3}, we get 
 \begin{equation}\label{Aditi_RM_Mistake_Equation_4}
 \begin{split}
 \|u_{m}(x,0)\|^{2}_{H^{2}\left(\Omega\right)}\leq C_{1}\left( u_{0},\Delta^{2} u_{m}(x,0)\right)=C_{1}\left( \Delta u_{0},\Delta u_{m}(x,0)\right)\\
 \leq C_{1}\left\|\Delta u_{m}(x,0)\right\|_{L^{2}\left(\Omega\right)}\,\left\|\Delta u_{0}\right\|_{L^{2}\left(\Omega\right)}.
 \end{split}
 \end{equation}
 Since $\left\|\Delta u_{m}(x,0)\right\|_{L^{2}\left(\Omega\right)}$ and $\|u_{m}(x,0)\|_{H^{2}\left(\Omega\right)}$ are equivalent norm, there exists $\alpha>0$ such that 
 $$\left\|\Delta u_{m}(x,0)\right\|_{L^{2}\left(\Omega\right)}\leq \alpha\,\|u_{m}(x,0)\|_{H^{2}\left(\Omega\right)}.$$
 Applying the fact $\left\|\Delta u_{m}(x,0)\right\|_{L^{2}\left(\Omega\right)}\leq \alpha\,\|u_{m}(x,0)\|_{H^{2}\left(\Omega\right)}$ in \eqref{Aditi_RM_Mistake_Equation_4}, we have
 \begin{equation}\label{Aditi_RM_Mistake_Equation_5}
 \|u_{m}(x,0)\|^{2}_{H^{2}\left(\Omega\right)}\leq C_{1}\alpha\,\|u_{m}(x,0)\|_{H^{2}\left(\Omega\right)}\,\left\|\Delta u_{0}\right\|_{L^{2}\left(\Omega\right)}.
 \end{equation} 
 Applying the inequality $pq\leq \beta\,p^{2}+\frac{1}{4\beta}q^{2}$ with $p=\|u_{m}(x,0)\|_{H^{2}\left(\Omega\right)}$, $q=\left\|\Delta u_{0}\right\|_{L^{2}\left(\Omega\right)}$ and $\beta=\frac{1}{2\,C_{1}\alpha}$, we get 
 \begin{equation}\label{Aditi_RM_Mistake_Equation_6}
 \|u_{m}(x,0)\|^{2}_{H^{2}\left(\Omega\right)}\leq\frac{1}{2}\,\|u_{m}(x,0)\|^{2}_{H^{2}\left(\Omega\right)} + \frac{C_{1}^{2}\,\alpha^{2}}{2}\,\left\|\Delta u_{0}\right\|^{2}_{L^{2}\left(\Omega\right)}.
 \end{equation}
 Therefore, we conclude \eqref{Aditi_RM_Mistake_equation_2} with $E=C_{1}\,\alpha$. 
\vspace{0.1cm}\\
Hence, we conclude that 
\begin{equation}\label{Aditi_rm_BV_equation10}
\begin{split}
\int_{\Omega}\norm{\frac{\partial u_m}{\partial t}(x,0)}dx\hspace{10 cm}
\\\leq \displaystyle\max_{1\leq i,j\leq d}\|B_{i,j}^{\prime}\|_{L^{\infty}\left(\mathbb{R}\right)}
\sum_{i,j=1}^d \displaystyle\max_{1\leq i,j\leq d}
\left\|\frac{\partial u_m}{\partial x_i}(x,0)\right\|_{L^{2}(\Omega)}
\left\|\frac{\partial u_m}{\partial x_j}(x,0)\right\|_{L^{2}\left(\Omega\right)}\\
+d\,\,E\,\,\displaystyle\max_{1\leq i,j\leq d}\|B_{ij}\|_{L^{\infty}\left(\mathbb{R}\right)}\,\,\left(\mbox{Vol}\left(\Omega\right)\right)^{\frac{1}{2}}\,\|u_{0}\|_{H^{2}\left(\Omega\right)}
\\+\displaystyle\max_{1\leq i\leq d}\|f_{i}^{\prime}\|_{L^{\infty}\left(\mathbb{R}\right)}\sum_{i=1}^d\,
\left\|\frac{\partial u_m}{\partial x_i}(x,0)\right\|_{L^{2}\left(\Omega\right)}\,\left(\mbox{Vol}\left(\Omega\right)\right)^{\frac{1}{2}}.
\end{split}
\end{equation}
from \eqref{Aditi_rm_BV_equation5}.\\
Applying the fact that $\|u_{m}(0)\|_{H^{1}_{0}\left(\Omega\right)}\leq \|u_{0}\|_{H^{1}_{0}\left(\Omega\right)}$ in \eqref{Aditi_rm_BV_equation10}, we get the existence of a constant $D_{2}>0$ such that for all $m\in\mathbb{N}$, the following 
\begin{equation}\label{Aditi_RM_BV_Equation_11}
\int_{\Omega}\norm{\frac{\partial u_m}{\partial t}(x,0)}dx\,\leq D
\end{equation}
holds. Integrating \eqref{Aditi_RM_BV_Equations_12} over the interval $[0,T]$ and using \eqref{Aditi_RM_BV_Equation_11}, we get
\begin{equation}\label{Aditi_RM_BV_Equation_12}
\left\|\frac{\partial u_m}{\partial t}\right\|_{L^1\left(\Omega_{T}\right)}\,\leq\,T\, D.
\end{equation} 
\vspace{0.1cm}\\
\textbf{Step-2:} Applying Energy estimates Theorem \ref{Energy_Estimates_Theorem_1}, we get a constant $D_{3}>0$ such that such that for all $m\in\mathbb{N}$, the following estimates
$$\|u_{m}\|_{L^{2}\left(0,T;\,H^{1}_{0}\left(\Omega\right)\right)}\leq D_{3}$$
holds. Using H\"{o}lder inequality, we now compute the following
\begin{eqnarray}\label{Aditi_RM_Equations_13}
\left\|\nabla\,u_{m}\right\|_{\left(L^{1}\left(\Omega_{T}\right)\right)^{d}}&=&\displaystyle\sum_{j=1}^{d}\int_{\Omega_{T}}\left|\frac{\partial u_{m}}{\partial x_{j}}\right|\,dx\,dt\nonumber\\
&\leq& \displaystyle\sum_{j=1}^{d}\left\|\frac{\partial u_{m}}{\partial x_{j}}\right\|_{L^{2}\left(\Omega_{T}\right)}\,\left(\mbox{Vol}\left(\Omega_{T}\right)\right)^{\frac{1}{2}}\nonumber\\
&\leq&\,d\,D_{3}\,\left(\mbox{Vol}\left(\Omega_{T}\right)\right)^{\frac{1}{2}}.
\end{eqnarray}
\vspace{0.1cm}\\
\textbf{Step-3:} Applying \textbf{Step-1} and \textbf{Step-2}, we conclude the proof of Theorem \ref{BV}.\\
\vspace{0.1cm}\\
In order to derive almost everywhere convergence, we state the following compact imbedding result from \cite{Godlewski_Raviart}.
\vspace{0.1cm}
\begin{theorem}\cite[p.53]{Godlewski_Raviart}\label{Aditi_RM_Compact_imbedding}
Let $\Omega$ be a bounded subset of $\mathbb{R}^{d}$ with a Lipschitz continuous boundary. Then the canonical imbedding of $\mbox{BV}\left(\Omega\right)$ into $L^{1}\left(\Omega\right)$ is compact.	
\end{theorem}
The next result gives the almost everywhere convergence of the Galerkin approximations.
\begin{theorem}\label{Aditi_RM_Compact_imbedding_result_1}
There exists a subsequence of the Galerkin approximations $\left(u_{m}\right)$ which converges almost everywhere to a function $u$ in $L^{1}\left(\Omega_{T}\right)$.	
\end{theorem}
\begin{proof}
In view of Theorem \ref{BV}, we get that the total variations of the Galerkin approximations $\left(u_{m}\right)$ are bounded. Therefore, an application of Theorem \ref{Aditi_RM_Compact_imbedding} with $\Omega=\Omega_{T}$ gives the existence of subsequence $\left(u_{m_{k}}\right)_{k=1}^{\infty}$ and a function $u$ in $L^{1}\left(\Omega_{T}\right)$ such that 
$$u_{m_{k}}\to\,u\,\,\mbox{as}\,\,k\to\infty$$
in $L^{1}\left(\Omega_{T}\right)$. Hence a further subsequence converges almost everywhere in $\Omega_{T}$. We rename the further subsequence of $\left(u_{m_{k}}\right)_{k=1}^{\infty}$	as $\left(u_{m}\right)$. This completes the proof of Theorem \ref{Aditi_RM_Compact_imbedding_result_1}. 
\end{proof}	
\vspace{0.1cm}\\
\subsection{Convergence}	\label{Convergence}
\textbf{Proof of Theorem \ref{Aditi_RM_Convergence_result_1}:} We prove Theorem \ref{Aditi_RM_Convergence_result_1} in two steps. In \textbf{Step-1}, we show that the almost every where limit $u$ of Galerkin approximations $\left(u_{m}\right)$ is a weak solution and in \textbf{Step-2}, we conclude the uniqueness of solutions in $W(0,T)$.\\
\vspace{0.1cm}\\
\textbf{Step-1: (Existence of a weak solution)}\\
\vspace{0.1cm}	
Since $u_{m}\to u$ in $L^{1}\left(\Omega_{T}\right)$ as $m\to\infty$, then for each $i\in\left\{1,2,\cdots,d\right\}$, we obtain
$$\frac{\partial u_{m}}{\partial x_{i}}\to\frac{\partial u}{\partial x_{i}}\,\,\mbox{as}\,\,m\to\infty\,\,\mbox{in}\,\,\mathcal{D}^{\prime}\left(\Omega_{T}\right).$$
An application of the Energy estimates Theorem \ref{Energy_Estimates_Theorem_1} gives that  for each $i\in\left\{1,2,\cdots,d\right\}$, we have
$$\frac{\partial u_{m}}{\partial x_{i}}\rightharpoonup v_{i}\,\,\mbox{as}\,\,m\to\infty\,\,\mbox{in}\,\, L^{2}\left(\Omega_{T}\right).$$
Hence, we have 
$$\frac{\partial u_{m}}{\partial x_{i}}\to v_{i}\,\,\mbox{as}\,\,m\to\infty\,\,\mbox{in}\,\, \mathcal{D}^{\prime}\left(\Omega_{T}\right).$$	
Therefore we conclude that 
$$v_{i}=\frac{\partial u}{\partial x_{i}}\,\,\mbox{in}\,\,\,\mathcal{D}^{\prime}\left(\Omega_{T}\right).$$
As a result, we have proved that $u$ lies in $L^{2}\left(0,T;H^{1}_{0}\left(\Omega_{T}\right)\right)$. Again, an application of Energy estimates Theorem \ref{Energy_Estimates_Theorem_1} gives the existence of a subsequence $\left(u_{m_{l}}\right)$ of $\left(u_{m}\right)$ such that   
$$u_{m_{l}}^{\prime}\rightharpoonup u^{\prime}\,\,\mbox{as}\,\,\,l\to\infty$$ 
in $L^{2}\left(0,T;H^{-1}\left(\Omega_{T}\right)\right)$. We rename the subsequence $\left(u_{m_{l}}\right)$ as $\left(u_{m}\right)$. As a result, we have that $u$ lies in $W(0,T)$.\\ We now show that $u$ satisfies \eqref{Aditi_RM_equation1_Defn} and  \eqref{Aditi_RM_equation2_Defn}. It is well known that the set 
$$ S:=\left\{\displaystyle\sum_{k=1}^{N}d^{k}(t)\,w_{k}\,\,:\,\,d^{1}(t),d^{2}(t),\cdots,d^{N}(t)\in\,C^{1}\left([0,T]\right)\right\}$$ 
is dense in $L^{2}\left(0,T;\,H^{1}_{0}\left(\Omega\right)\right)$. For $N\leq m$, let $v:=\displaystyle\sum_{k=1}^{N}d^{k}(t)\,w_{k}$. Observe that $u_{m}$ satisfies 
\begin{equation}\label{Aditi_Ramesh_Paper1_Final}
\int_{0}^{T}\langle u_{m}^{\prime},v\rangle\,dt ~+\,\displaystyle\sum_{i=1}^{d}\displaystyle\sum_{j=1}^{d}\int_{0}^{T}\left(\frac{\partial u_{m}}{\partial x_{j}},\,B_{i,j}(u_{m})\,\frac{\partial v}{\partial x_{i}}\right)\,dt  + \displaystyle\sum_{i=1}^{d}\int_{0}^{T}\left(\frac{\partial u_{m}}{\partial x_{i}},\,f^{\prime}_{i}\left(u_{m}\right)\,v\right)\,dt=0.
\end{equation}  
Applying Theorem \ref{Aditi_RM_Compact_imbedding_result_1} and passing to the limit as $m\to\infty$ in \eqref{Aditi_Ramesh_Paper1_Final}, we conclude that $u$ satisfies
\begin{equation}\label{Aditi_Ramesh_Paper1_Final_1}
\int_{0}^{T}\langle u^{\prime},v\rangle\,dt +\,\displaystyle\sum_{i=1}^{d}\displaystyle\sum_{j=1}^{d}\int_{0}^{T}\left(B_{i,j}(u)\frac{\partial u}{\partial x_{j}}\,,\frac{\partial v}{\partial x_{i}}\right)\,dt  + \displaystyle\sum_{i=1}^{d}\int_{0}^{T}\left(f^{\prime}_{i}\left(u\right)\frac{\partial u}{\partial x_{i}},\,\,v\right)\,dt=0.
\end{equation}
Since $S$ is dense in $L^{2}\left(0,T;H^{1}_{0}\left(\Omega\right)\right)$, for all $v$ in $L^{2}\left(0,T;H^{1}_{0}\left(\Omega\right)\right)$, we have
\begin{equation}\label{Aditi_Ramesh_Paper1_Final_2}
\int_{0}^{T}\langle u^{\prime},v\rangle\,dt +\,\displaystyle\sum_{i=1}^{d}\displaystyle\sum_{j=1}^{d}\int_{0}^{T}\left(B_{i,j}(u)\frac{\partial u}{\partial x_{j}}\,,\frac{\partial v}{\partial x_{i}}\right)\,dt  + \displaystyle\sum_{i=1}^{d}\int_{0}^{T}\left(f^{\prime}_{i}\left(u\right)\frac{\partial u}{\partial x_{i}},\,\,v\right)\,dt=0.
\end{equation}
We conclude that $u$ satisfies \eqref{Aditi_RM_equation1_Defn} and  \eqref{Aditi_RM_equation2_Defn} by following similar arguements of Theorem 2.1 in \cite{Ramesh}. Therefore $u$ is a weak solution of the problem \eqref{eq1.1}, \eqref{eq1.2}, \eqref{eq1.3}.\\
\vspace{0.1cm}\\
\textbf{Proof of Theorem \ref{Aditi_RM_Convergence_result_2}:} Let $v \in H^1_0(\Omega)$ such that $\nrm{v}_{H^1_0(\Omega)} \le 1$. We get the following
$$\left|\inp{\frac{\partial u_m}{\partial t}}v\right|
\!\!\leq E\,\|u_{m}\|_{H^{1}_{0}\left(\Omega\right)}$$
from \eqref{POIAOSC_Change_Eqn1} of \textbf{Step-4} in Energy Estimates result Theorem \ref{Energy_Estimates_Theorem_1}. Applying Gelfand triplet \cite[p.135-p.136]{Brezis}, we conclude that
\begin{equation}\label{Aditi_RM_Theorem_2_Eqn1}
\left|\left(\frac{\partial u_m}{\partial t},\,v\right)\right|
\!\!\leq E\,\|u_{m}\|_{H^{1}_{0}\left(\Omega\right)}.
\end{equation} 
Denote 
$$ S:=\left\{v\in L^{2}\left(\Omega\right)\,:\,v \in H^1_0(\Omega)\,\,\mbox{and}\, \nrm{v}_{H^1_0(\Omega)} \le 1\right\}.$$
In view of \eqref{Aditi_RM_Theorem_2_Eqn1}, we obtain that the map $A:\,H^{1}_{0}\left(\Omega\right)\to\mathbb{R}$ given by
$$A(v)=\left(v,\,\frac{\partial u_m}{\partial t}\right)$$
is a continuous linear functional $H^{1}_{0}\left(\Omega\right)$ in $L^{2}\left(\Omega\right)$ norm for each $t\in\left(0,T\right)$ and $m\in\mathbb{N}$. We know that $H^{1}_{0}\left(\Omega\right)$ is dense in $L^{2}\left(\Omega\right)$. Then, we extend $A$ from $H^{1}_{0}\left(\Omega\right)$ to a unique continuous linear functional $\tilde{A}$ on $L^{2}\left(\Omega\right)$ in $L^{2}\left(\Omega\right)$ norm by applying Hahn Banach extension theorem. Therefore, we conclude that
$$\tilde{A}(v)=\left(v,\,\frac{\partial u_{m}}{\partial t}\right)$$
on $L^{2}\left(\Omega\right)$. Then an application of Riesz representation theorem gives that $$\|\tilde{A}\|=\left\|\frac{\partial u_{m}}{\partial t}\right\|_{L^{2}\left(\Omega\right)}.$$
In view of \eqref{Aditi_RM_Theorem_2_Eqn1} and Energy Estimates result Theorem \ref{Energy_Estimates_Theorem_1}, we conclude that there exist a constant $C_{1}>0$ and $C_{2}>0$ such that
\begin{equation}\label{Aditi_RM_Equations_1}
\left\|\frac{\partial u_{m}}{\partial t}\right\|_{L^{2}\left(0,T;L^{2}\left(\Omega\right)\right)}\leq C_{1}\,\|u_{m}\|_{L^{2}\left(0,T;H^{1}_{0}\left(\Omega\right)\right)}\leq C_{2}\|u_{0}\|_{L^{2}\left(\Omega\right)}.
\end{equation}
Therefore, the total variation of the Galerkin approximations are given by
\begin{equation}\label{Aditi_RM_POIAOSC_Eqn2}
TV_{\Omega_{T}}\left(u_{m}\right)=\int_{\Omega_{T}}\left|\frac{\partial u_{m}}{\partial t}\right|\,dx\,dt + \displaystyle\sum_{j=1}^{d}\int_{\Omega_{T}}\left|\frac{\partial u_{m}}{\partial x_{j}}\right|\,dx\,dt.
\end{equation}
Applying Energy Estimates result Theorem \ref{Energy_Estimates_Theorem_1}, \eqref{Aditi_RM_Equations_1} and H\"{o}lder inequality, we conclude that the sequence $\left(TV_{\Omega_{T}}(u_{m})\right)$ is bounded. Then the extraction of almost everywhere convergent subsequence of $\left(u_{m}\right)$ follows from applications of similar arguments as used in Theorem \ref{Aditi_RM_Compact_imbedding_result_1}. Then the proof of the existence of weak solutions is obtained from the proof of Theorem \ref{Aditi_RM_Convergence_result_1} with \textbf{Hypothesis-H2}. This completes the proof.
\vspace{0.2cm}\\
\textbf{Step-2: (Uniqueness of weak solutions)}\\ Let $u_{1}$ and $u_2$ be solutions of \eqref{eq1.1}, \eqref{eq1.2}, \eqref{eq1.3} in $W(0,T)$. Denote $w:=u_{1}-u_{2}$. We apply a uniqueness result for linear parabolic problem from \cite[p.150]{Ladyzenskaja} to the linear equation satisfied by $w$ to conclude the proof of uniqueness of weak solutions in $W(0,T)$.  \\
For $i,j\in\left\{1,2,\cdots,d\right\}$, applications of Fundamental theorem of calculus give
\begin{equation}\label{Aditi_Ramesh_Paper1_Final_3}
\begin{split}
B_{ij}(u_{1})\frac{\partial u_{1}}{\partial x_{j}}-B_{ij}\left(u_{2}\right)\frac{\partial u_{2}}{\partial x_{j}} = w\,\int_{0}^{1}B_{ij}^{\prime}\left(\tau\,u_{1} +\left(1-\tau\right)\,u_{2}\right)\,\left[\tau\,\left(u_{1}\right)_{x_{j}}+\left(1-\tau\right)\,\left(u_{2}\right)_{x_{j}}\right]\,d\tau \\+ \frac{\partial w}{\partial x_{j}}\,\int_{0}^{1}\,B_{ij}\left(\tau\,u_{1}+\left(1-\tau\right)u_{2}\right)\,d\tau
\end{split}
\end{equation} 
and 
\begin{equation}\label{Aditi_Ramesh_Paper1_Final_4}
f_{i}(u_{1})-f_{i}(u_{2})=w\,\int_{0}^{1}\,f_{i}^{\prime}\left(\tau\,u_{1} +\left(1-\tau\right)u_{2}\right)\,d\tau.
\end{equation}
For $i,j\in\left\{1,2,\cdots,d\right\}$, denote
\begin{eqnarray}\nonumber\\
\widetilde{a}_{i,j}(x,t)&:=& \int_{0}^{1}\,B_{ij}\left(\tau\,u_{1}+\left(1-\tau\right)u_{2}\right)\,d\tau,\nonumber\\
\widetilde{b}_{ij}(x,t) &:=& \int_{0}^{1}B_{ij}^{\prime}\left(\tau\,u_{1} +\left(1-\tau\right)\,u_{2}\right)\,\left[\tau\,\left(u_{1}\right)_{x_{j}}+\left(1-\tau\right)\,\left(u_{2}\right)_{x_{j}}\right]\,d\tau,\nonumber\\
\widetilde{c}_{i}(x,t) &:=& \int_{0}^{1}\,f_{i}^{\prime}\left(\tau\,u_{1} +\left(1-\tau\right)u_{2}\right)\,d\tau.\nonumber
\end{eqnarray}
Using equations \eqref{Aditi_Ramesh_Paper1_Final_2}, \eqref{Aditi_Ramesh_Paper1_Final_3} and \eqref{Aditi_Ramesh_Paper1_Final_4}, we have that $w$ satisfies 
\begin{equation}\label{Aditi_Ramesh_Paper1_Final_5}
\int_{0}^{T}\langle w^{\prime},v\rangle\,dt +\,\displaystyle\sum_{i=1}^{d}\displaystyle\sum_{j=1}^{d}\int_{0}^{T}\left(\widetilde{a}_{i,j}(x,t)\frac{\partial w}{\partial x_{j}} +\,\widetilde{b}_{ij}(x,t)\,w\,,\frac{\partial v}{\partial x_{i}}\right)\,dt  - \displaystyle\sum_{i=1}^{d}\int_{0}^{T}\widetilde{c}_{i}(x,t)\left(w,\frac{\partial v}{\partial x_{i}}\right)\,dt=0
\end{equation}
for all $v\in\,H^{1}\left(\Omega\times (0,T)\right)$ such that $v(x,T)=0$ on $\Omega$ and $v=0$ on $\partial\Omega\times(0,T)$.\\
Therefore, $w$ is a weak solution of the following linear parabolic equation
$$\mathcal{L}(w):= w_{t}-\displaystyle\sum_{i=1}^{d}\displaystyle\sum_{j=1}^{d}\frac{\partial}{\partial x_{i}}\left(\widetilde{a}_{i,j}(x,t)\frac{\partial w}{\partial x_{j}} + \,\widetilde{b}_{ij}(x,t)\,w\right) + \displaystyle\sum_{i=1}^{d}\frac{\partial}{\partial x_{i}}\left(\widetilde{c}_{i}(x,t)\,w\right)=0$$
such that $w(x,t)=0$ on $\partial\Omega\times (0,T)$ and $w(x,0)=0$ in $\Omega$.
A proof of uniqueness of problem \eqref{eq1.1}, \eqref{eq1.2}, \eqref{eq1.3} follows by applying similar arguments as used in \cite[p.150]{Ladyzenskaja}.\\
\vspace{0.2cm}\\
\textbf{Proof of Theorem \ref{Aditi_RM_Convergence_result_3}:} A proof of Theorem \ref{Aditi_RM_Convergence_result_3} follows by applying similar arguments as used in \cite[p.60]{Godlewski_Raviart}.
\bibliographystyle{plain}


\end{document}